\documentclass{amsart}

\usepackage[margin=2cm]{geometry}

\usepackage{accents,amsrefs,amssymb,bm,rsfso,tikz}
\usepackage[linktocpage=true, colorlinks=true, linkcolor={red!75!blue}, citecolor={green!75!black}, urlcolor=blue, pdfencoding=auto, psdextra]{hyperref}

\usetikzlibrary{arrows,math}

\newtheorem{Theorem}{Theorem}[section]
\newtheorem{Proposition}[Theorem]{Proposition}
\newtheorem{Corollary}[Theorem]{Corollary}
\newtheorem{Lemma}[Theorem]{Lemma}

\theoremstyle{definition}
\newtheorem{Definition}[Theorem]{Definition}
\newtheorem{Remark}[Theorem]{Remark}

\newcommand\ph\varphi
\renewcommand\le\leqslant
\renewcommand\ge\geqslant
\newcommand\then\Rightarrow
\newcommand\into\hookrightarrow
\newcommand\otno\twoheadleftarrow
\newcommand\onto\twoheadrightarrow
\newcommand\impl{\,{\to}\,}
\newcommand\bimp{\,{\leftrightarrow}\,}
\newcommand\mb\mathbb
\newcommand\sD{\mathcal D}
\newcommand\sF{\mathcal F}
\newcommand\KM{\ensuremath{\bm{\mathsf{KM}}}}
\newcommand\IPC{\ensuremath{\bm{\mathsf{IPC}}}}
\newcommand\set[1]{\left\{#1\right\}}
\newcommand\setof[2]{\left\{#1\vphantom{#2}\,\right|\left.\vphantom{#1}#2\right\}}
\newcommand\brk[1]{\left\langle#1\right\rangle}
\newcommand{\dbtilde}[1]{\accentset{\approx}{#1}}

\title{An explicit Kuznetsov-Muravitsky enrichment}
\author{Mamuka Jibladze}
\email{mamuka.jibladze@tsu.ge}
\author{Evgeny Kuznetsov}
\address{Department of Mathematical Logic\\
Razmadze Mathematical Institute, TSU\\
Tbilisi 0193, Georgia}
\email{e.kuznetsov@freeuni.edu.ge}
\thanks{Research supported by the SRNSF Grant \#FR-22-6700}

\begin{document}

\maketitle

\section{Introduction}

Around 1978, A.~V.~Kuznetsov introduced the now famous modal extension of the Intuitionistic Propositional Calculus \IPC,
later thoroughly investigated by him and his devoted collaborator A.~Yu.~Muravitsky in a series of papers.
We will use for their system the by now well-established name \KM.

In the posthumously published paper \cite{Kuznetsov} Kuznetsov proved that, for any \IPC-formul\ae\ $A$, $B$,
\begin{equation*}\label{theokuz}
\KM+A\vdash B \iff\IPC+A\vdash B.
\tag{K}
\end{equation*}
His proof is essentially proof-theoretic, based on an inductive elimination of the modality from inferences.

On the other hand, from the very beginning Kuznetsov and Muravitsky studied their calculus in parallel with its algebraic semantics,
via what they called \emph{$\Delta$-pseudoboolean algebras}. In this paper we will call them \KM-algebras, following Leo Esakia,
who also has an important contribution to their study \cite{Esakia}.

A \KM-algebra is a Heyting algebra $(H,\land,\lor,\impl,0,1)$ together with a unary operation $\Delta$ satisfying the identities
\begin{align*}
x&\le\Delta x\\
\Delta x\impl x&=x\\
\Delta x&\le y\lor y\impl x.
\end{align*}

As established by Kuznetsov and Muravitsky, \eqref{theokuz} is equivalent to the statement that every variety of Heyting algebras
is generated by reducts of \KM-algebras. In fact they also showed that \eqref{theokuz} implies existence of an embedding of any Heyting algebra $H$
into a \KM-algebra generating the same variety of Heyting algebras as $H$.

As stressed by several people, including Muravitsky himself, it would be highly desirable to have a purely algebraic construction of such an embedding.
In \cite{Muravitsky} he provides one such construction, but his proof that the ambient algebra stays in the same variety is again essentially proof-theoretic,
based on a thorough analysis of derivations in various auxiliary modal intuitionistic calculi.

In this paper, an entirely different explicit embedding is described, manifestly remaining in the same variety.
We do not know whether the extended algebra obtained by us is isomorphic to the one from \cite{Muravitsky}.

Our approach is straightforward: given an element $a\in H$, we adjoin to $H$ a new element $\Delta(a)$ with desired properties freely \emph{in the variety of $H$},
obtaining an algebra $H{\brk{\Delta(a)}}$ with a homomorphism $H\to H{\brk{\Delta(a)}}$. In order to show that this homomorphism is an embedding and does not spoil any other $\Delta(b)$ that might already exist in $H$,
we give an explicit description of the algebra $H{\brk{\Delta(a)}}$. After that we proceed similarly to \cite{Muravitsky}, showing that one can consistently do this for every $a\in H$,
obtaining an embedding $H\to H{\brk\Delta}$ such that every $a\in H$ possesses a $\Delta(a)$ in $H{\brk\Delta}$.
Then, again as in \cite{Muravitsky} we iterate to obtain a countable chain of embeddings $H\to H{\brk\Delta}\to H{\brk\Delta}{\brk\Delta}\to\cdots$ and finally take the union
(that is, direct limit) of this chain, which then turns out to be a \KM-algebra. Thus the only difference of our construction from that in \cite{Muravitsky} is at the very first step.
The point is that, unlike what happens in \cite{Muravitsky}, it is immediate that this first step does not leave the variety of $H$, and this then obviously remains true at each iteration too.

In the next section we introduce necessary notation and gather some facts about Heyting and \KM-algebras that we will need.
In \S\ref{main}, the crucial step of our construction is described, and in \S\ref{final} we explain how to use it to achieve the desired embedding.
The last section addresses the relationship between our construction and that of \cite{Muravitsky}.

\subsection*{Acknowledgements}
The authors are grateful to Nick Bezhanishvili, Dato Gabelaia and Kote Razmadze for numerous helpful suggestions.

\section{Preliminaries}\label{prelim}

For an \IPC-formula $\ph(p_1,...,p_n)$, where $\{p_1,...,p_n\}$ is the set of all variables of $\ph$, and for elements $h_1,...,h_n\in H$ of a Heyting algebra $H$,
we will denote by $\ph(h_1,...,h_n)$ the value of $\ph$ under the valuation (in the algebraic semantics) interpreting the variable $p_i$ as $h_i$, $i=1,...,n$.

We will need the following
\begin{Proposition}\label{maintool}
Let $\ph(p_0,p_1,...,p_n)$ be any \IPC-formula, and let $h$, $x$, $x'$, $h_1$,..., $h_n$ be any elements of a Heyting algebra $H$. Then,
\[
\text{$h\land x=h\land x'$ \emph{implies} $h\land\ph(x,h_1,...,h_n)=h\land\ph(x',h_1,...,h_n)$.}
\]
\end{Proposition}
\begin{proof}
This follows easily from the fact that
\[
(p\bimp p')\impl(\ph(p,p_1,...,p_n)\bimp\ph(p',p_1,...,p_n))
\]
is a theorem of \IPC.
\end{proof}

\begin{Proposition}\label{elem}
For elements $a,d\in H$ of a Heyting algebra $H$, the following are equivalent:
\begin{itemize}
\item $d$ is \emph{dense over $a$}, that is, $a\le d$ and $d\impl a=a$
\item $d\impl a\le d$
\item $d=h\lor h\impl a$ for some $h\in H$.
\end{itemize}
\end{Proposition}

This is well known; see e.~g. \cite{Muravitsky}*{Proposition 7.2}.

For an element $a\in H$ of a Heyting algebra $H$ we denote by $\sD_a(H)$ the subset of $H$ consisting of all elements dense over $a$. That is,
\[
\sD_a(H)=\setof{d\in H}{\text{$a\le d$ and $d\impl a=a$}}.
\]
It is well known (and straightforward to check) that $\sD_a(H)$ is a filter.

We recall from \cite{Muravitsky}:

\begin{Proposition}
A Heyting algebra $H$ is a reduct of a \KM-algebra if and only if for every $a\in H$ the filter $\sD_a(H)$ is principal.
\end{Proposition}

In fact it is easy to see that the \KM-axioms precisely enforce $\Delta a$ to be the smallest element of $\sD_a(H)$ for every $a\in H$.

Using this, we will tacitly denote the smallest element of $\sD_a(H)$, when it exists, by $\Delta_H(a)$.
Moreover, we will say ``$\Delta_H(a)$ exists in $H$'' to mean ``$\sD_a(H)$ is principal'', and ``$x=\Delta_H(a)$'' to mean ``$x$ is the smallest element of $\sD_a(H)$''.

We will also need
\begin{Lemma}\label{eqlemma}
For any Heyting algebra $H$ and any $x,h,a\in H$,
\[
x\impl((h\impl a)\impl h) = ((x\impl h)\impl a)\impl(x\impl h).
\]
\end{Lemma}
\begin{proof}
We transform the right hand side as follows:
\begin{align*}
((x\impl h)\impl a)\impl(x\impl h)
&=(((x\impl h)\impl a)\land x)\impl h\\
&=((((x\impl h)\land x)\impl a)\land x)\impl h\\
&=(((h\land x)\impl a)\land x)\impl h\\
&=((x\impl(h\impl a))\land x)\impl h\\
&=((h\impl a)\land x)\impl h\\
&=x\impl((h\impl a)\impl h).
\end{align*}
\end{proof}

\begin{Corollary}\label{eqD}
For any Heyting algebra $H$ and any $x,h,a\in H$, one has $x\impl h\in\sD_a(H)$ if and only if $x\le(h\impl a)\impl h$.
\end{Corollary}
\begin{proof}
The conditions to be proved equivalent correspond to the left, resp. right hand side of the equality in Lemma \ref{eqlemma} to be $1$.
\end{proof}

\begin{Proposition}\label{future}
Let $f:H\to H'$ be a homomorphism of Heyting algebras, let $\sF$ be a filter of $H$, and let $\brk{f(\sF)}$ be the filter of $H'$ generated by $f(\sF)$.
Then, if $\sF$ is principal, $\brk{f(\sF)}$ is principal too, and if $b$ is the smallest element of $\sF$ then $f(b)$ is the smallest element of $\brk{f(\sF)}$.
\end{Proposition}
\begin{proof}
Explicitly, $\brk{f(\sF)}=\setof{h'\in H'}{\exists h\in\sF\ f(h)\le h'}$.
Clearly $f(b)\in\brk{f(\sF)}$, and if $f(h)\le h'$ for some $h\in\sF$ then $f(b)\le f(h)\le h'$.
\end{proof}

\begin{Corollary}\label{presdelta}
For $a\in H$, let $f:H\to H'$ be a homomorphism such that for any $d'\in\sD_{f(a)}(H')$ there is a $d\in\sD_{a}(H)$ with $f(d)\le d'$.
Then, if $\Delta_H(a)$ exists, also $\Delta_{H'}(f(a))$ exists and $\Delta_{H'}(f(a))=f(\Delta_H(a))$.
\end{Corollary}
\begin{proof}
The hypothesis just means that $\sD_{f(a)}(H')=\brk{f(\sD_a(H))}$, so this follows from Proposition \ref{future} above.
\end{proof}

\begin{Proposition}\label{nasty}
An element $a'\in H$ satisfies $a'=\Delta_H(a)$ (i.~e. $a'$ is the smallest element of $\sD_a(H)$) if and only if
\[
(a\impl h)\land((h\impl a)\impl a) = a'\impl h
\]
for all $h\in H$.
\end{Proposition}
\begin{proof}
The `if' part is clear: the left hand side is $1$ if and only if $h\in\sD_a(H)$ while the right hand side is $1$ if and only if $a'\le h$.

For the `only if' part, note that by Proposition \ref{elem} and Corollary \ref{eqD}, $x\le$ the left hand side iff $x\impl h\in\sD_{a}(H)$.
On the other hand, $x\le$ the right hand side iff $a'\le x\impl h$. And if $a'=\Delta(a)$ then these two must indeed be equivalent. 
\end{proof}

We will also need a fact from universal algebra in general. For it we recall the concept of extension by constants from an algebra.

\begin{Definition}\label{vaa}
Given an algebra $A$ in some variety $\mb V$, we denote by $\mb V_A$ the variety obtained by extending the signature of $\mb V$
with the set of constants (i.~e. operations of arity zero) $\setof{c_a}{a\in A}$ and adding to the identities of $\mb V$
the set of identities $t_1(c_{a_1},\dots,c_{a_n})=t_2(c_{a_1},\dots,c_{a_n})$
for every pair of terms $t_1$, $t_2$ of the same arity $n$ in the signature of $\mb V$ and every $a_1,\dots,a_n\in A$ such that $t_1(a_1,\dots,a_n)=t_2(a_1,\dots,a_n)$.

Moreover we denote by $\mb V_A(A)$ the subvariety of $\mb V_A$ generated by the algebra $(A,(a)_{a\in A})$.
That is, we view $A$ as a $\mb V_A$-algebra via interpreting the constant $c_a$ as $a$ for each $a\in A$.
\end{Definition}

Note that there is a one-to-one correspondence between arbitrary $\mb V_A$-algebras $(A',(c_a)_{a\in A})$
and $\mb V$-algebras $A'$ equipped with a $\mb V$-homomorphism $f:A\to A'$.
Indeed given $f:A\to A'$ we can add values of the constants $c_a$ to $A'$ through $c_a:=f(a)$;
conversely, for any $\mb V_A$-algebra $(A',(c_a)_{a\in A})$ we get a map $f:A\to A'$ given by $f(a):=c_a$.
Obviously under this correspondence, $f$ is a homomorphism iff the $\mb V_A$-identities are satisfied in $(A',(c_a)_{a\in A})$.
In particular, to the above algebra $(A,(a)_{a\in A})$ corresponds the identity homomorphism $A\to A$.

By definition, a $\mb V$-algebra $A'$ belongs to the subvariety $\mb V(A)$ of $\mb V$ generated by $A$
iff there are homomorphisms $A'\otno S\into A^I$ (that is, a set $I$, a subalgebra $S\subseteq A^I$ and an onto homomorphism $S\onto A'$).
In case of $\mb V_A$,

\begin{Proposition}\label{obvious}
If a $\mb V_A$-algebra $(A',(c_a)_{a\in A})$ is in $\mb V_A(A)$ then $A'$ is in $\mb V(A)$.
Conversely, for a $\mb V(A)$-algebra $A'$, a choice of constants $(c_a)_{a\in A}$, $c_a\in A'$,
turns $(A',(c_a)_{a\in A})$ into a $\mb V_A(A)$-algebra iff there is a subalgebra $S\subseteq A^I$
and an onto homomorphism $q:S\onto A'$ such that $S$ contains the diagonal subalgebra $A\subseteq A^I$
and $q(a)=c_a$ for each $a\in A$.
\end{Proposition}
\begin{proof}
Immediately from the above correspondence between $\mb V_A$-algebras and homomorphisms,
$(A',(c_a)_{a\in A})$ is in $\mb V_A(A)$ iff there is a $\mb V_A$-subalgebra $S$ of the $\mb V_A$-algebra
$(A\to A)^I=(A\to A^I)$ and an onto $\mb V_A$-homomorphism $q:S\onto A'$.
Obviously the above just means that $S$ contains the diagonal, and $q$ maps the value of $a\in A$ under the diagonal to the constant $c_a\in A'$.
\end{proof}

\begin{Corollary}\label{general}
Let $f:A\to B$ and $g:B\to C$ be homomorphisms of algebras in some variety $\mb V$.
Then by the above, we have a $\mb V_A$-algebra $(B,(f(a))_{a\in A})$ and a $\mb V_B$-algebra $(C,(g(b))_{b\in B})$.

If $(B,(f(a))_{a\in A})$ is in $\mb V_A(A)$ and $(C,(g(b))_{b\in B})$ is in $\mb V_B(B)$ then also $(C,(g(f(a)))_{a\in A})$ is in $\mb V_A(A)$.
\end{Corollary} 
\begin{proof}
By Proposition \ref{obvious} the hypothesis is equivalent to the data given $A\subseteq S\subseteq A^I$ and $q:S\onto B$ with $q|_A=f$
as well as $B\subseteq T\subseteq B^J$ and $r:T\onto C$ with $r|_B=g$.

Given that, let $U\subseteq S^J$ be the preimage of $T\subseteq B^J$ under $q^J:S^J\onto B^J$, so that $q^J|_U$ is an onto homomorphism $s:U\onto T$.
Then $U$ contains the diagonal $S\subseteq S^J$ and $A\subseteq A^J\subseteq S^J\subseteq A^{I\times J}$,
thus $U$ also contains the diagonal $A\subseteq A^{I\times J}$.

Moreover $q|_A=f$ and $r|_B=g$ implies $(r\circ s)|_{A\subseteq S\subseteq U}=g\circ f$.

The whole argument can be summarized in a diagram as follows:
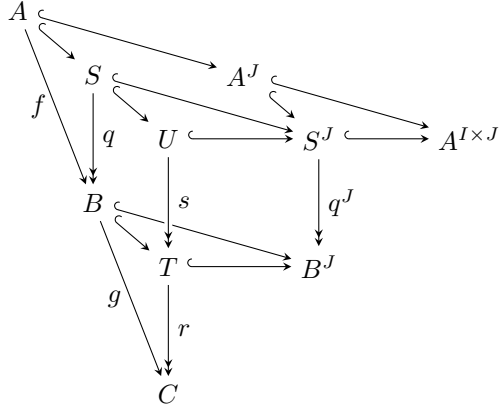
\begin{center}
\begin{tikzpicture}[xscale=2,yscale=1.7,>=stealth,evaluate={\tr=.5;}]
\node (C) at (0,0) {$C$};
\node (T) at (0,1) {$T$};
\node (U) at (0,2) {$U$};
\node (BJ) at (1,1) {$B^J$};
\node (SJ) at (1,2) {$S^J$};
\node (AIJ) at (2,2) {$A^{I\times J}$};
\node (B) at (-\tr,1+\tr) {$B$};
\node (S) at (-\tr,2+\tr) {$S$};
\node (A) at (-2*\tr,2+2*\tr) {$A$};
\node (AJ) at (1-\tr,2+\tr) {$A^J$};
\draw[->] (A) -- node[left] {$f$} (B);
\draw[->] (B) -- node[left] {$g$} (C);
\draw[right hook->] (A) -- (AJ);
\draw[right hook->] (AJ) -- (AIJ);
\draw[right hook->] (A) -- (S);
\draw[right hook->] (S) -- (U);
\draw[right hook->] (U) -- (SJ);
\draw[right hook->] (SJ) -- (AIJ);
\draw[right hook->] (B) -- (T);
\draw[right hook->] (T) -- (BJ);
\draw[right hook->] (B) -- (BJ);
\draw[right hook->] (S) -- (SJ);
\draw[right hook->] (AJ) -- (SJ);
\draw[line width=3pt,white] (U) -- (T);
\draw[->>] (U) -- node[right] {$s$} (T);
\draw[->>] (T) -- node[right] {$r$} (C);
\draw[->>] (SJ) -- node[right] {$q^J$} (BJ);
\draw[->>] (S) -- node[right] {$q$} (B);
\end{tikzpicture}
\end{center}
\end{proof}

\section{The one-step enrichment}\label{main}

\begin{Theorem}\label{onestep}
For any Heyting algebra $H$ and any $a\in H$, there is a set $D$, a subalgebra $H{\brk\iota}\subseteq H^D$ containing the diagonal $H\subseteq H^D$, a Heyting algebra $H{\brk{\Delta(a)}}$ and a surjective homomorphism $\pi:H{\brk\iota}\onto H{\brk{\Delta(a)}}$ such that
\begin{itemize}
\item[(a)]\label{onesteprin} the filter $\sD_{\pi(a)}(H{\brk{\Delta(a)}})$ is principal;
\item[(b)]\label{onestepinc} the composite homomorphism $H\hookrightarrow H{\brk\iota}\onto H{\brk{\Delta(a)}}$ is one-to-one;
\item[(c)]\label{onestepres} for any $b\in H$ with $\sD_b(H)$ principal, the filter $\sD_{\pi(b)}(H{\brk{\Delta(a)}})$ is principal too.
\end{itemize}
\end{Theorem}
\begin{proof}
We take $D=\sD_a(H)$. Thus let $H^{\sD_a(H)}$ be the product of $\sD_a(H)$ many copies of $H$. It will be convenient to view it as the algebra of all maps $\sD_a(H)\to H$ with pointwise operations.

We will identify $H$ with the diagonal subalgebra in $H^{\sD_a(H)}$. Or else, the subalgebra of all constant maps.

Note that $\sD_a(H^{\sD_a(H)})=\sD_a(H)^{\sD_a(H)}$. Indeed $\delta\in\sD_a(H^{\sD_a(H)})$ by definition means that $a\le\delta(h)$ and $\delta(h)\impl a=a$ for any $h\in H$, which precisely means that $\delta(\sD_a(H))\subseteq\sD_a(H)$.

Let $H{\brk\iota}\subseteq H^{\sD_a(H)}$ be the subalgebra generated by $H\subseteq H^{\sD_a(H)}$ and one more element, namely the identical embedding
\[
\iota:\sD_a(H)\hookrightarrow H,
\]
$\iota(d)\equiv d$. Note that $\iota\in\sD_a(H{\brk\iota})$ ($=\sD_a(H)^{\sD_a(H)}\cap H{\brk\iota}$) since $a\le d=\iota(d)$ and $\iota(d)\impl a=d\impl a=a$ for all $d\in\sD_a(H)$.

Let $\sF_a$ be the filter of $H{\brk\iota}$ generated by the subset $\setof{\iota\impl\delta}{\delta\in\sD_a(H{\brk\iota})}$ of $H{\brk\iota}$. Note that the generating subset is closed under finite meets, so that
\[
\sF_a=\setof{\gamma\in H{\brk\iota}}{\exists\delta\in\sD_a(H{\brk\iota})\ \iota\impl\delta\le\gamma}.
\]

Let $H{\brk{\Delta(a)}}=H{\brk\iota}/\sF_a$. We will denote the class of an $\eta\in H{\brk\iota}$ in $H{\brk{\Delta(a)}}$ with $[\eta]$.

Clearly if $\delta\in\sD_a(H{\brk\iota})$ then $[\delta]\in\sD_{[a]}(H{\brk{\Delta(a)}})$. But also conversely, $[\eta]\in\sD_{[a]}(H{\brk{\Delta(a)}})$ implies $\eta\in\sD_a(H{\brk\iota})$ for any $\eta\in H{\brk\iota}$.

Indeed by definition $[\eta]\in\sD_{[a]}(H{\brk{\Delta(a)}})$ means that
there is some $\delta\in\sD_a(H{\brk\iota})$ with $\iota\impl\delta\le(a\impl\eta)\land((\eta\impl a)\impl a)$.
Thus $d\impl\delta(d)\le a\impl\eta(d)$ and $d\impl\delta(d)\le(\eta(d)\impl a)\impl a$ for all $d\in\sD_a(H)$.
Then $a=a\land(d\impl\delta(d))\le\eta(d)$, and
\[
\eta(d)\impl a = ((\eta(d)\impl a)\impl a)\impl a\le(d\impl\delta(d))\impl a = a,
\]
so $\eta\in\sD_a(H{\brk\iota})$.

It follows that for all $x$ in $H{\brk{\Delta(a)}}$, $x\in\sD_{[a]}(H{\brk{\Delta(a)}})$ iff $[\iota]\le x$.

Hereby \ref{onesteprin}(a) is proved.

To prove \ref{onestepinc}(b) we will need the following

\begin{Proposition}\label{partool}
For any $\eta\in H{\brk\iota}$, any $x,y\in\sD_a(H)$ and any $h\in H$,
\[
h\land x=h\land y \then h\land\eta(x)=h\land\eta(y).
\]
\end{Proposition}
\begin{proof}
By construction of $H{\brk\iota}$, each $\eta\in H{\brk\iota}$ has form $\ph(\iota,h_1,...,h_n)$ for some \IPC-formula $\ph$. That is, one can find a formula $\ph(p_0,p_1,...,p_n)$  and $h_1$, ..., $h_n\in H$ such that $\eta(d)=\ph(d,h_1,...,h_n)$ for all $d\in\sD_a(H)$. We thus can use Proposition \ref{maintool} to reach the desired conclusion.
\end{proof}

\begin{Corollary}\label{findone}
For any $\delta\in\sD_a(H{\brk\iota})$, 
\[
\delta(1)\le\delta(\delta(1)).
\]
\end{Corollary}
\begin{proof}
Taking in Proposition \ref{partool} $h=x=d$ and $y=1$ we obtain
\[
d\land\eta(d)=d\land\eta(1)
\]
for any $d\in\sD_a(H)$ and any $\eta\in H{\brk\iota}$. If $\eta=\delta\in\sD_a(H{\brk\iota})$ then $\delta(1)\in\sD_a(H)$, hence we can take $d=\delta(1)$ and obtain
\[
\delta(1)\land\delta\delta(1)=\delta(1)\land\delta(1)=\delta(1),
\]
so indeed $\delta(1)\le\delta\delta(1)$.
\end{proof}

Returning to \ref{onestepinc}(b), what we have to show is that the composite $H\hookrightarrow H{\brk\iota}\twoheadrightarrow H{\brk\iota}/\sF_a=H{\brk{\Delta(a)}}$ remains one-to-one. By definition this means that $\iota\impl\delta\le h$ implies $h=1$, for any $\delta\in\sD_a(H{\brk\iota})$ and any $h\in H$.

Indeed $\iota\impl\delta\le h$ means that $d\impl\delta(d)\le h$ for all $d\in\sD_a(H)$. Taking here $d=\delta(1)$, by Corollary \ref{findone} we obtain $1=\delta(1)\impl\delta(\delta(1))\le h$, i.~e. $h=1$.

Finally to prove \ref{onestepres}(c), suppose given $b\in H$ such that $\sD_b(H)$ is principal,
i.~e. there is a $c\in H$ such that $b\le c$, $c\impl b=b$, and $c\le x$ for any $x\in H$ with $b\le x$ and $x\impl b=b$.
Then clearly $[b]\le[c]$ and $[c]\impl[b]=[b]$ in $H{\brk{\Delta(a)}}$, so we have to prove that $[c]\le x$ for any $x\in\sD_{[b]}(H{\brk{\Delta(a)}})$,
i.~e. for any $x\in H{\brk{\Delta(a)}}$ with $[b]\le x$ and $x\impl[b]=[b]$. Given such $x$, take some $\eta\in H{\brk\iota}$ with $x=[\eta]$.
Then, $x\in\sD_{[b]}(H{\brk{\Delta(a)}})$ iff there is a $\delta\in\sD_a(H{\brk\iota})$ with $\iota\impl\delta\le (b\impl\eta)\land((\eta\impl b)\impl b)$, i.~e.
\[
d\impl\delta(d)\le(b\impl\eta(d))\land((\eta(d)\impl b)\impl b)
\]
for all $d\in\sD_a(H)$.

By Proposition \ref{nasty} the hypotheses imply $(b\impl h)\land((h\impl b)\impl b)=c\impl h$ for any $h\in H$.
Taking here $h=\eta(d)$ we conclude that $d\impl\delta(d)\le c\impl\eta(d)$ for all $d\in\sD_a(H)$.
In other words, $\iota\impl\delta\le c\impl\eta$, which means that $[c]\le[\eta]=x$ in $H{\brk{\Delta(a)}}$.
\end{proof}

\begin{Remark}\label{stronger}
Note that, since $H{\brk\iota}\subseteq H^{\sD_a(H)}$ contains the diagonal, by Proposition \ref{obvious} we have actually proved that
$H\hookrightarrow H{\brk{\Delta(a)}}$ belongs to the variety $\mb V_H(H)$.
\end{Remark}

\begin{Remark}
There is an alternative argument for \ref{onestepres}(b): using Proposition \ref{partool} one can easily prove that
\[
\iota\impl\eta=\iota\impl\eta(1)
\]
for any $\eta\in H{\brk\iota}$. It then also follows that the filter $\sF_a$ is already generated by only the elements of the form $\iota\impl d$ for $d\in\sD_a(H)$.

Clearly then $h\in\sF_a$ iff $\iota\impl d\le h$ for some $d\in\sD_a(H)$, and then $\iota(d)\impl d=d\impl d=1\le h$.
\end{Remark}

\begin{Remark}
Let us note that shifting from $H^D$ to $H{\brk\iota}$ is essential. For example, consider
\[
H=\set{0<\cdots<\frac1n<\cdots<\frac13<\frac12<\frac11=1}.
\]
In this algebra $d\in\sD_0(H)$ iff $d>0$.

For some $n_0>0$, let $\delta_0:\sD_0(H)\to H$ be the map given by
\[
\delta_0\left(\frac1n\right)=\frac1{n+n_0}.
\]
Then $\iota\impl\delta_0=\delta_0\le\frac1{n_0}$, so if instead of $H{\brk\iota}$ we would take any algebra $H'$ with $H{\brk\iota}\subseteq H'\subseteq H^{\sD_0(H)}$ such that $\delta_0\in H'$, then for the analog of $\sF_0$, i.~e. for the filter $\sF'$ of $H'$ generated by $\setof{\iota\impl\delta}{\delta\in\sD_0(H')}$ we would get that $\frac1{n_0}\in\sF'$, so that the composite $H\into H'\onto H'/\sF'$ would not be an embedding.
\end{Remark}

\section{The embedding}\label{final}

We are going to perform the described construction for all elements of a given Heyting algebra.
To show this can be done in a consistent way, we will characterize this construction by a universal property.

\begin{Proposition}\label{extend}
Let $f:H\to H'$ be a homomorphism of Heyting algebras such that the $\mb V_H$-algebra $(H',(f(h))_{h\in H})$ belongs to the variety $\mb V_H(H)$ (notation from Definition \ref{vaa}).
For $a\in H$, suppose that $H'$ possesses $\Delta(f(a))$, i.~e. the filter $\sD_{f(a)}(H')$ is principal, with smallest element $\Delta(f(a))$.
Then, there is a unique extension of $f$ to a homomorphism $\tilde f:H{\brk{\Delta(a)}}\to H'$ along the inclusion $H\into H{\brk{\Delta(a)}}$ satisfying $\tilde f(\Delta(a))=\Delta(f(a))$.
\end{Proposition}
\begin{proof}
Uniqueness is clear as $H{\brk{\Delta(a)}}$ is generated by $H$ and $\Delta(a)$.

For existence, let us additionally consider the subalgebra $H{\brk i}$ of $H^H$
generated by the diagonal $H\subseteq H^H$ together with the identity map $i:H\to H$.

We claim that $H{\brk i}$ has the following universal property:
for any homomorphism $f:H\to H'$ representing an algebra in $\mb V_H(H)$ and any $h'\in H'$,
there is a unique extension $\tilde f:H{\brk i}\to H'$ of $f$ along $H\subseteq H{\brk i}$ with $\tilde f(i)=h'$.

In other words, $H{\brk i}$ is a free $\mb V_H(H)$-algebra on one generator $i$.

Indeed by definition of $H{\brk i}$ each of its elements has form $\ph(i,h_1,...,h_n)$ for some \IPC-formula $\ph(p_0,p_1,...,p_n)$ and some $h_1$, ..., $h_n$.
That is, viewed as a map $H\to H$, it sends $h\in H$ to $\ph(h,h_1,...,h_n)$.

Thus uniqueness is clear: we are forced to put
\[
\tilde f(\ph(i,h_1,...,h_n))=\ph(h',f(h_1),...,f(h_n)).
\]

For existence, we have to show that for any $\ph(p_0,p_1,...,p_n)$, whenever $\ph(i,h_1,...,h_n)=1$ in $H{\brk i}$ then also $\ph(h',f(h_1),...,f(h_n))=1$ in $H'$.

Indeed $\ph(i,h_1,...,h_n)=1$ means that $\ph(h,h_1,...,h_n)=1$ for all $h\in H$, i.~e. $\ph(x,c_{h_1},...,c_{h_n})$ is an identity in the variety $\mb V_H(H)$,
consequently $\ph(h',f(h_1),...,f(h_n))=1$ in $H'$ as $f:H\to H'$ is in $\mb V_H(H)$.
 
Next observe that restriction along the embedding $\sD_a(H)\subseteq H$ induces a surjective homomorphism $\rho:H^H\onto H^{\sD_a(H)}$ such that $\rho(H{\brk i})=H{\brk\iota}$.
Moreover, $\rho(i)=\iota$, and the above extension $\tilde f:H{\brk i}\to H'$ factors through $\rho$ iff $h'\in\sD_{f(a)}(H')$. Indeed if such a factorization exists then $h'$ is the image of $\iota$ under this factorization, and $\iota\in\sD_a(H{\brk\iota})$, so $h'\in\sD_{f(a)}(H')$. For the converse, combining Proposition \ref{maintool} with Corollary \ref{eqD} shows that, given an \IPC-formula $\ph(p_0,p_1,...,p_n)$ and some $h_1,...,h_n\in H$, the equality $\ph(d,h_1,...,h_n)=1$ holds for any $d\in\sD_a(H)$ if and only if
\[
(h\impl a)\impl h\le\ph(h,h_1,...,h_n)
\]
for any $h\in H$. This then means that $(x\impl a)\impl x\le\ph(x,c_{h_1},...,c_{h_n})$ is an identity of $\mb V_H(H)$. So since $f:H\to H'$ is in $\mb V_H(H)$, for any $h'\in H'$ one has $(h'\impl f(a))\impl h'\le\ph(h',f(h_1),...,f(h_n))$. In particular, if $h'\in\sD_{f(a)}(H')$ then $(h'\impl f(a))\impl h'=1$, consequently $\ph(h',f(h_1),...,f(h_n))=1$.

Moreover if $h'=\Delta(f(a))$ then $\tilde f(\iota\impl\delta)=h'\impl\tilde f(\delta)=1$ for any $\delta\in\sD_a(H{\brk\iota})$ since $\tilde f(\sD_a(H{\brk\iota}))\subseteq\sD_{f(a)}(H')$,
hence $\tilde f$ factors further through $H{\brk\iota}\onto H{\brk{\Delta(a)}}$.
\end{proof}

\begin{Corollary}\label{iso}
For any $a,b\in H$, there is a unique isomorphism between the algebras $H{\brk{\Delta_H(a)}}{\brk{\Delta_{H{\brk{\Delta_H(a)}}}(b)}}$ and $H{\brk{\Delta_H(b)}}{\brk{\Delta_{H{\brk{\Delta_H(b)}}}(a)}}$ which identifies $\Delta_{H{\brk{\Delta_H(b)}}}(a)$ with $\Delta_H(a)$, $\Delta_{H{\brk{\Delta_H(b)}}}(b)$ with $\Delta_H(b)$, and fixes $H$.
\end{Corollary}
\begin{proof}
By Remark \ref{stronger} and Corollary \ref{general}, the composite homomorphism
\[
f:H\into H{\brk{\Delta_H(b)}}\into H{\brk{\Delta_H(b)}}{\brk{\Delta_{H{\brk{\Delta_H(b)}}}(a)}}
\]
is in $\mb V_H(H)$, so by Proposition \ref{extend} there is a unique extension
\[
\tilde f:H{\brk{\Delta_H(a)}}\to H{\brk{\Delta_H(b)}}{\brk{\Delta_{H{\brk{\Delta_H(b)}}}(a)}}
\]
of $f$ along $H\into H{\brk{\Delta_H(a)}}$ sending $\Delta_H(a)$ to $\Delta_{H{\brk{\Delta_H(b)}}}(a)$. Then again by the same Proposition, using \ref{onestepres}(c) we obtain that this $\tilde f$ further uniquely extends along the embedding $H{\brk{\Delta_H(a)}}\into H{\brk{\Delta_H(a)}}{\brk{\Delta_{H{\brk{\Delta_H(a)}}}(b)}}$ to a
\[
\dbtilde f:H{\brk{\Delta_H(a)}}{\brk{\Delta_{H{\brk{\Delta_H(a)}}}(b)}}\to H{\brk{\Delta_H(b)}}{\brk{\Delta_{H{\brk{\Delta_H(b)}}}(a)}}
\]
sending $\Delta_{H{\brk{\Delta_H(a)}}}(b)$ to $\Delta_H(b)$.

Interchanging $a$ and $b$, we similarly obtain a unique homomorphism
\[
\dbtilde g:H{\brk{\Delta_H(b)}}{\brk{\Delta_{H{\brk{\Delta_H(b)}}}(a)}}\to H{\brk{\Delta_H(a)}}{\brk{\Delta_{H{\brk{\Delta_H(a)}}}(b)}}
\]
extending the composite $g:H\into H{\brk{\Delta_H(a)}}\into H{\brk{\Delta_H(a)}}{\brk{\Delta_{H{\brk{\Delta_H(a)}}}(b)}}$ along the above $f$
such that $\dbtilde g(\Delta_H(b))=\Delta_{H{\brk{\Delta_H(a)}}}(b)$ and $\dbtilde g(\Delta_{H{\brk{\Delta_H(b)}}}(a))=\Delta_H(a)$.

Finally by uniqueness we obtain that both composites $\dbtilde f\circ\dbtilde g$ and $\dbtilde g\circ\dbtilde f$ are identities.
\end{proof}

Repeatedly using Corollary \ref{iso} we obtain that for any finite subset $S\subseteq H$ there is an embedding $H\subseteq H{\brk{\Delta(S)}}$, defined uniquely up to a unique isomorphism,
with $H{\brk{\Delta(S)}}$ possessing $\Delta(a)$ for all $a\in S$, and moreover for any two such subsets $S,T$ with $S\subseteq T$ there is a unique one-to-one homomorphism $H{\brk{\Delta(S)}}\into H{\brk{\Delta(T)}}$ preserving $H$ and all $\Delta(a)$ for $a\in S$.

We thus obtain a directed system of embeddings of algebras $H{\brk{\Delta(S)}}$ in the variety of $H$, indexed by all finite subsets of $H$.
Let $H{\brk\Delta}$ be the direct limit of this system. It is well known (see e.~g. \cite{Gratzer}*{Chapter 3, Exercises 34, 35}) that the limit of a directed diagram of algebras in any variety stays in the same variety and moreover that if all transition maps in a directed diagram are embeddings, then all of the induced homomorphisms to the limit are embeddings too. We thus obtain that $H{\brk\Delta}$ is in the variety of $H$ and all induced homomorphisms $H{\brk{\Delta(S)}}\to H{\brk\Delta}$ are embeddings.
In particular, the resulting map $H\to H{\brk\Delta}$ is also an embedding.

We note that $H{\brk\Delta}$ possesses $\Delta_{H{\brk\Delta}}(a)$ for all $a\in H$. Indeed any $\delta\in\sD_{H{\brk\Delta}}(a)$ is an element of some $H{\brk{\Delta(S)}}$,
hence also of $H{\brk{\Delta(S\cup\set{a})}}$, and belongs to $\sD_{H{\brk{\Delta(S\cup\set{a})}}}(a)$ there, thus $\Delta_{H{\brk{\Delta(S\cup\set{a})}}}(a)\le\delta$.
Consequently, the image of $\Delta_{H{\brk{\Delta(S\cup\set{a})}}}(a)$ has the needed properties for $\Delta_{H{\brk\Delta}}(a)$.

Next, iterating this construction, we obtain another directed system, now an $\omega$-chain, $H=H_0\into H_1\into H_2\into\cdots$,
with $H_{n+1}=H_n{\brk\Delta}$.

Let us denote the limit of this chain by $H^{\KM}$. We then have
\begin{Theorem}
The algebra $H^{\KM}$ is a \KM-algebra belonging to the variety of $H$, and $H$ embeds into it.
\end{Theorem}
\begin{proof}
Any element $\alpha\in H^{\KM}$ lies in some $H_n$, so $H_{n+1}$ possesses $\Delta_{H_{n+1}}(\alpha)$. Then arguing as above for $H{\brk\Delta}$,
we see that the image of $\Delta_{H_{n+1}}(\alpha)$ in $H^{\KM}$ has the defining properties of $\Delta_{H^{\KM}}(\alpha)$. The rest is clear from the construction.
\end{proof}

\section{Relation to Muravitsky's construction}

In \cite{Muravitsky}, at the first step the following construction is used.
Let us denote the Stone embedding of $H$ into the Heyting algebra
of all upper sets of the poset (with the inclusion order) $\mathrm{Spec}(H)$ of all prime filters of $H$ by $\sigma:H\to\mathrm{Up}(\mathrm{Spec}(H))$, i.~e. $\sigma(h)=\setof{\mathfrak p\in\mathrm{Spec}(H)}{h\in\mathfrak p}$ for $h\in H$.

For each $a\in H$, Muravitsky uses the embedding $H\into\delta[H_a]$, where $\delta[H_a]$ is the subalgebra of $\mathrm{Up}(\mathrm{Spec}(H))$
generated by $\sigma(H)$ and one more element
\[
\sigma(a)_+:=\sigma(a)\cup\max(\mathrm{Spec}(H)\setminus\sigma(a)).
\]

As mentioned in the introduction, his proof that $\delta[H_a]$ belongs to the variety of $H$ is essentially proof-theoretic,
and it is desirable to relate his $\delta[H_a]$ to our $H{\brk{\Delta(a)}}$.

We do not have at our disposal a homomorphism $H{\brk\iota}\to\mathrm{Up}(\mathrm{Spec}(H))$ extending $\sigma$ along $H\into H{\brk\iota}$
and sending $\iota$ to $\sigma(a)_+$, since $\mathrm{Up}(\mathrm{Spec}(H))$ might be outside the variety of $H$.
Such homomorphism, with image $\delta[H_a]$ (and factoring through $H{\brk\iota}\onto H{\brk{\Delta(a)}}$) would indeed exist, if the following statement were true.

\ 

For any \IPC-formula $\ph(p_0,p_1,...,p_n)$ and any $h_1, ...,h_n\in H$ with
\[
\ph(\sigma(a)_+,\sigma(h_1),...,\sigma(h_n))\ne\mathrm{Spec}(H)
\]
there is a $d\in H$ with $\sigma(a)_+\subseteq\sigma(d)$ and $\ph(\sigma(d),\sigma(h_1),...,\sigma(h_n))\ne\mathrm{Spec}(H)$.

\ 

We do not know whether this holds.

\end{document}